\documentclass[11pt]{article}
\usepackage[margin=1.2in]{geometry}
\usepackage{graphicx}
\usepackage{authblk}
\usepackage[dvipsnames]{xcolor}
\usepackage{amsmath,amssymb,amsfonts,amsthm}
\usepackage[ruled,vlined]{algorithm2e}
\usepackage{caption}
\usepackage{subcaption}
\DeclareCaptionLabelFormat{fullsubfigure}{\thefigure#2}
\usepackage{enumitem}
\usepackage{url}

\usepackage[
    style=numeric-comp,
    hyperref=true,
    doi=false,
    url=false,
    isbn=false,
    giveninits=true,
    sorting=none,
    block=none,
    backend=bibtex,
    maxnames=99
]{biblatex}
\renewbibmacro{in:}{}
\usepackage{hyperref}

\newtheorem{theorem}{Theorem}
\newtheorem{lemma}{Lemma}

\newtheorem{definition}{Definition}
\newtheorem{remark}{Remark}

\newcommand{\bE}{\mathbb{E}}
\newcommand{\N}{\mathbb{N}}

\newcommand{\bS}{\mathbb{S}}

\newcommand{\R}{\mathbb{R}}

\newcommand{\bfo}{\mathbf{1}}

\newcommand{\tr}{\operatorname{tr}}

\makeatletter
\def\blfootnote{\xdef\@thefnmark{}\@footnotetext}
\makeatother

\title{Feedback Stabilization for Sampled Linear Systems with Control-linear Noise}

\begin{document}
\author{Cynthia Cheng \quad \mbox{and} \quad Xudong Chen}
\date{}
\maketitle
\blfootnote{C. Cheng and X. Chen are with the Department of Electrical and Systems Engineering, Washington University, St. Louis. Emails: \texttt{\{cheng.cynthia,cxudong\}@wustl.edu}. Corresponding author: C. Cheng.}

\begin{abstract}
In this paper, we consider linear stochastic systems with control-linear noise and periodically sampled measurements. 
We address the problem of feedback stabilization in the mean-square sense. The main contribution of the paper is to provide a necessary and sufficient condition for feedback stabilization. In particular, we relate feedback stabilizability of the stochastic system to the existence of a fixed point of a Riccati-type algebraic equation and, further, to the existence of a solution to an infinite-horizon optimal control problem.   
\end{abstract}

\section{Introduction}\label{sec:intro}
We consider in this paper a continuous-time, linear time-invariant stochastic system in the It{\^ o} sense, with control-linear noise and periodically sampled measurements:
\begin{equation}\label{eq:systemmodel}
\left \{
\begin{aligned}
d x_t & = A x_t dt + B u(t) (dt + \sigma dw_t), \\
 y_{k\tau} & = x_{k\tau}, \quad \mbox{for } k  = 0,1,2,\cdots,
\end{aligned}
\right.  
\end{equation}
where $x_t\in \R^n$ is the state of the system, $u(t)\in \R^m$ is the control input, $\sigma$ is a real number,  
$w_t$ is the standard Wiener process, and $y_{k\tau}$ is the measurement at time instant $k\tau$. The sampling period is $\tau > 0$, and the sampling rate is $1/\tau$.  We assume that the initial condition $x_0$ is given, but can be arbitrary.
We address the problem of feedback stabilization in the mean-square sense for system~\eqref{eq:systemmodel}. A precise problem formulation will be given shortly at the beginning of Section~\ref{sec:mainresults}.  

Stochastic systems with state- and/or control-dependent noise have widely been appreciated in the literature for their use in modeling systems with human operators~\cite{levison1969model}, sensorimotor systems~\cite{todorov2005stochastic}, 
mechanical systems subject to random vibrations~\cite{bratus2018optimal}, electric propulsion engines which experience thrust
uncertainties that are linearly proportional to the level of commanded thrusts~\cite{jenson2020optimal}, just to name a few. 

The problem of feedback stabilization for these stochastic systems is not new and, in fact, has been extensively studied in the literature. Often one relates feedback stabilizability of the system to the existence of a solution to the associated infinite-horizon optimal control problem, which can further be translated into the existence of a positive definite solution of a generalized algebraic Riccati equation. See, e.g.,~\cite{kleinman1969optimal,McLane1971,willems1976feedback} for continuous-time linear systems and~\cite{de1982infinite,ito2016linear,ito2023stochastic} for discrete-time linear systems. In these works, the authors have assumed that one can access full-state, perfect measurements at {\it any} time instant or time step, and looked for linear feedback control laws $u(t) = Kx_t$, for constant~$K$, to stabilize the systems.   

What makes this paper different from these existing works is the hybrid setting where the dynamics evolve in continuous time while the measurements are taken at discrete time steps. Note that system~\eqref{eq:systemmodel} is fully parameterized by the quadruple $(A, B, \sigma^2, \tau)$. In particular, the sampling rate $1/\tau$ emerges as a new parameter in the hybrid setting. Our motivation for considering such a setting is rooted in astronautical applications, especially in autonomous spacecraft guidance, navigation, and control, where the near-orbit dynamics of the spacecraft can be modeled by linear control systems with additive noise, control-dependent noise, sampled measurements, and impulsive and/or continuous control inputs~\cite{jenson2022robust}. Our model~\eqref{eq:systemmodel} can naturally be applied to a broader class of cyber-physical systems wherever the dynamics of the physical system are continuous while digital communications, sensing, etc., are scheduled a priori or sporadic (e.g., event driven).  

The same hybrid setting and its variations have been considered in the earlier work~\cite{jenson2020optimal,jenson2021optimal,jenson2022robust}, where a class of finite-horizon optimal control problems have been posed and solved. To the best of our knowledge, feedback stabilization for~\eqref{eq:systemmodel} has not yet been addressed in the literature. 

It is intuitive that feedback stabilizability of system~\eqref{eq:systemmodel}, with $(A, B)$ fixed, is a monotone property with respect to $\sigma^2$ and~$\tau$. Indeed, the larger $\sigma^2$ is, the more uncertainty the control input $u(t)$ will bring into the system. It is also clear that the larger the sampling rate $1/\tau$ is, the more samples the controller will obtain per unit of time.   
The above arguments, as well as the simulation results we obtain in the numerical study (see Subsection~\ref{ssec:numericalstudy}), motivate us to investigate the interplay between $\sigma^2$ and $\tau$.  

We are driven by the desire to uncover the fundamental limit of the sampling rate that can sustain feedback stabilization of system~\eqref{eq:systemmodel}.  Specifically, we ask    
\begin{description}
    \item[Question:] {\it Given $(A, B)$, what is the supremum of $\tau$ (resp. $\sigma^2$) for a given $\sigma^2$ (resp. $\tau$) such that~\eqref{eq:systemmodel} is feedback stabilizable?}  
\end{description}
A stepping stone toward a complete answer to this question is to obtain   necessary and sufficient conditions for feedback stabilizability of system~\eqref{eq:systemmodel}, through which one wishes to establish connections between feedback stabilization and problems of other types, thus enabling the use of tools from various research domains. The main results of this paper, which we present in Section~\ref{sec:mainresults}, serve the above purpose. In the limit $\tau\to 0$, our results recover the necessary and sufficient condition~\cite[Theorem~1]{willems1976feedback} for the continuous-time setting.

\section{Main Results}\label{sec:mainresults}
We start by introducing the class of feedback control laws that will be considered in the paper. 
A matrix-valued function $K: [0,\tau)\to \R^{m\times n}$ is said to be an $\mathrm{L}^2$-function if $$\int_0^\tau \tr(K^\top (s)K(s)) ds < \infty,$$ where $\tr(\cdot)$ is the trace of a square matrix. 
Let $\mathrm{L}^2([0,\tau), \R^{m\times n})$ be the space of all such functions. For the purpose of feedback stabilization, we consider the following class of linear feedback control laws:   
\begin{equation}\label{eq:feedbackcontrollaw}
u(t):= K(t-k\tau) x_{k\tau},   
\end{equation}
where $K\in \mathrm{L}^2([0,\tau), \R^{m\times n})$, $t \in [k\tau, (k+1)\tau )$, and $k\in \N_0$ (throughout this paper, we use $\N_0$ to denote the set of nonnegative integers).  
Now, we have 

\begin{definition}
    System~\eqref{eq:systemmodel} is {\bf mean-square feedback stabilizable} if there exists a feedback gain $K\in \mathrm{L}^2([0,\tau), \R^{m\times n})$ such that for any initial condition $x_0\in \R^n$, the solution $x_t$ of the system driven by the feedback control law $u(t)$ given in~\eqref{eq:feedbackcontrollaw}
    satisfies  
    $\lim_{t\to\infty}\bE [\|x_t\|^2]  = 0$. 
    We call any such $K$ a {\bf stabilizing feedback gain}.  
\end{definition}

We present the main results of the paper in the next two subsections. Their proofs will be given in Section~\ref{sec:proofs}. 

\subsection{Necessary and sufficient condition}\label{ssec:necessaryandsufficientcondition}
Let $\bS_n$ be the space of $n$-by-$n$ symmetric matrices, and $\bS^+_n$ be the cone of $n$-by-$n$ positive semidefinite matrices. 
We present below a necessary and sufficient condition for  system~\eqref{eq:systemmodel} to be feedback stabilizable.
The condition concerns the existence of a fixed point of a map $f: \mathbb{S}_n^+\to \mathbb{S}_n^+$, which we introduce below. 
The map~$f$ turns out to be the update rule for the value-to-go function associated with the dynamic programming for solving a relevant stochastic optimal control problem. We will come back to this in Subsections~\ref{ssec:optimalcontrol} and~\ref{ssec:finitehorizon}.

To this end, let $(A, B, \sigma, \tau)$ be the parameter tuple of system~\eqref{eq:systemmodel}. Let $L\in \R^{n\times n}$ and $R\in \R^{m\times m}$ be positive definite matrices. 
Given $W\in \bS^+_n$, 
we introduce the following auxiliary matrices:
\begin{equation}\label{eq:defmatrixMandQ}
\left\{
\begin{aligned}
    Q(s) & := \sigma^2 B^\top e^{A^\top  s} W e^{A s} B + R, \\
    M & := \int_0^\tau e^{A s} B Q^{-1}(s) B^\top e^{A^\top s} ds. 
\end{aligned}
\right. 
\end{equation}
It is clear that $Q(s)> 0$. The matrix $M$ is similar in form to the controllability Gramian associated with $(A, B)$  and is positive semidefinite. We then define ${f:\mathbb{S}_n^+\to \mathbb{S}_n^+}$ as follows:
\begin{equation}\label{eq:deffunctionf}
    f(W) := e^{A^\top \tau} W^{\frac{1}{2}} (I + W^{\frac{1}{2}}MW^{\frac{1}{2}})^{-1} W^{\frac{1}{2}}e^{A \tau} + \tau L. 
\end{equation}
Note that $f$ depends implicitly on $L$ and $R$. 
Further, we call a matrix $W\in \bS^+_n$ a {\it fixed point} of $f$ if it satisfies
\begin{equation*}\label{eq:genralizedalgebraicriccati}
W = f(W). 
\end{equation*} 
Since $L > 0$, a fixed point of $f$ is necessarily positive definite. 

We state below relevant properties of the map $f$.  

\begin{theorem}\label{thm:globalconvergence}
    Let $f$ be given as in~\eqref{eq:deffunctionf}, with $L$ and $R$ positive definite matrices. 
    If $f$ has a fixed point~$W$, then it is unique. Moreover, for any $P\in \bS^+_n$, 
    $\lim_{N\to\infty} f^N(P) = W$. 
\end{theorem}

The next result relates mean-square feedback stabilizability to the existence of a fixed point of~$f$. 

\begin{theorem}\label{thm:feedbackstabilizableandw} 
    The following two items hold for system~\eqref{eq:systemmodel}: 
    \begin{enumerate}
        \item If there exist $L > 0$ and $R > 0$ such that the map $f$ given in~\eqref{eq:deffunctionf} has a (unique) fixed point $W$, then system~\eqref{eq:systemmodel} is mean-square feedback stabilizable. Moreover, the map $K: [0,\tau)\to \R^{m\times n}$ given by
        \begin{equation}\label{eq:stabilizingfeedbackgain}
            K(s)  :=   - Q^{-1}(\tau - s) B^\top e^{- A^\top s} (W - \tau L), 
        \end{equation}
        is a stabilizing feedback gain, where $Q$ is defined in~\eqref{eq:defmatrixMandQ}. 
        \item Conversely, if system~\eqref{eq:systemmodel} is mean-square feedback stabilizable, then for any $L> 0$ and $R> 0$, $f$ has a (unique) fixed point.  
    \end{enumerate}
\end{theorem}

\begin{remark}
By the first-order analysis in~$\tau$, we can write
\begin{equation*}
f(W) = W + \tau(A^\top W+WA  -WB(R+\sigma^2B^\top WB)^{-1}B^\top W+L) + o(\tau).
\end{equation*}
Thus, in the asymptotic regime $\tau \to 0$, the condition that $f(W)=W$ is related to the condition that $W$ solves the following algebraic equation: 
$$
A^\top W+WA -WB(R+\sigma^2B^\top WB)^{-1}B^\top W+L = 0.
$$
The latter is known~\cite[Theorem~1]{willems1976feedback} to be necessary and sufficient for feedback stabilizability of the continuous-time counterpart of system~\eqref{eq:systemmodel}, i.e., with $y_t = x_t$ for all $t \geq 0$.  
\end{remark}

\subsection{Connections with stochastic optimal control}\label{ssec:optimalcontrol}

There is a natural connection between mean-square feedback stabilization and stochastic optimal control, as we elaborate below.   
To proceed, we first relax the class of linear feedback control laws by allowing for heterogeneous feedback gains. 
A linear feedback control law $u: [0,\infty) \to \R^m$ is said to be {\it admissible} if it takes the following form: 
$$u(t) = K_k(t - k\tau) x_{k\tau},$$
where $K_k\in \mathrm{L}^2([0,\tau), \R^{m\times n})$ for all~$k\in \N_0$ and $t\in [k\tau, (k + 1)\tau)$.  
We note, without a proof, that if $u(t)$ is admissible, then both $\bE[x_t]$ and $\bE[x_tx_t^\top]$ are finite for all $t\in [0,\infty)$. 
For convenience, we use $\mathcal{U}$ to denote the set of admissible linear feedback control laws. Further, for each $N\in \N_0$, we let 
$$\mathcal{U}_N:= \{u |_{[0,N\tau)} \mid u \in \mathcal{U}\}.$$

Next, consider a family of finite-horizon optimal control problems, parameterized by the horizon $N\tau$ for $N \in \N_0$. Specifically, we define the cost function as  

\begin{equation}\label{eq:finitehorizon}
J(x, u;N):=
\bE \Bigg [\tau \sum_{k = 0}^{N-1} x^\top_{k\tau}\, L \, x_{k\tau}
+ \int_0^{N\tau} u^\top(t) R u(t) dt
+ x_{N\tau}^\top\, P \, x_{N\tau} \mid x_0 = x\Bigg ],
\end{equation}
where  $L, R> 0$ and $P \geq 0$. We allow $P$ to be zero (in this case, the terminal cost $\mathbb{E}[x_{N\tau}^\top P x_{N\tau}]$ is~$0$). Then, the finite-horizon optimal control problem is given by
    \begin{equation}\label{eq:finiteoptimal}
        \min_{u\in \mathcal{U}_N} J(x, u; N) \quad \mbox{subject to system~\eqref{eq:systemmodel}}.
    \end{equation}
We have the following result:

\begin{theorem}\label{thm:finitehorizon} 
    Let $K_k(s;N)$ and $W_k(N)$ be given in Algorithm~\ref{alg:gainmatrix}.  
    Then, for any $x \in \R^n$, the optimal control problem~\eqref{eq:finiteoptimal} 
    has a unique minimizer ${u}^*_N\in \mathcal{U}_N$, which is given by
\begin{equation*}
    {u}^*_N(t) := K_{k}(t-k\tau;N) {x}_{k\tau},   
\end{equation*} 
for $t\in [k\tau, (k+1)\tau)$ and $k = 0,\ldots, N-1$. 
The minimized cost is given by  
\begin{equation*}
    J^*(x; N) := J(x, u^*_N; N) = {x}^\top  W_0(N) {x}=  x^\top f^N(P) x. 
\end{equation*}
\end{theorem} 

\begin{algorithm}
\caption{Compute the gain matrices}
\label{alg:gainmatrix}
	\SetAlgoLined
	\KwIn{Matrices $A$, $B$, $L$, $R$, $P$, scalars $\sigma$ and $\tau$, and the horizon $N\tau$.}
	\KwOut{Matrices $W_k(N)$ and $K_k(s;N)$ for all $k = 0,\ldots, N-1$.}
	Initialize
	\begin{equation*}
	W_{N}(N) := P.
	\end{equation*}
	\For{$k := N - 1$ to $0$}{
		\renewcommand{\arraystretch}{1.3}
		\begin{equation*}
		\begin{aligned}
		Q_k(s;N) & :=  R +\sigma^2 B^\top e^{A^\top s} W_{k + 1}(N) e^{A s}B, \\ 
		W_{k}(N) & :=  f(W_{k+1}(N)).
		\end{aligned} 
		\end{equation*}
		
		Obtain $K_k(s;N): [0,\tau)\to \R^{m\times n}$ by  
		\begin{equation*}
		K_{k}(s;N)  :=   - Q^{-1}_k(\tau - s;N) B^\top e^{- A^\top s} (W_k(N) - \tau L). 
		\end{equation*}
	}
\end{algorithm}

Note that the functions $K_k(s;N)$ are uniformly bounded (and continuous) and hence belong to $\mathrm{L}^2([0,\tau), \R^{m\times n})$, so $u^*_N$ indeed belongs to $\mathcal{U}_N$. Also, note that $W_0(N) = f^N(P)$.

Theorem~\ref{thm:finitehorizon} ties $f^N(P)$ to the minimized cost $J^*(x;N)$. By letting $N\to\infty$, we  relate the limit $W := \lim_{N\to\infty} f^N(P)$, if it exists, to the cost  
of the resulting infinite-horizon problem. 
Note that $W$ is necessarily a fixed point of~$f$. Also, note that the feedback gain $K_0(s; N)$ in Algorithm~1 converges to $K(s)$ given in~\eqref{eq:stabilizingfeedbackgain}. 
We formalize the result below. First, let 
\begin{equation}\label{eq:costfunction}
J(x, u):=  \bE\left [ \tau \sum_{k = 0}^\infty x_{k\tau}^\top \, L\,  x_{k\tau} + \int_0^\infty u^\top(t) R u(t) d t \mid x_0 = x\right ],
\end{equation}
where $L > 0$ and $R > 0$. The infinite-horizon optimal control problem is then given by
\begin{equation}\label{eq:infiniteoptimal}
        \min_{u\in \mathcal{U}} J(x, u)  \quad \mbox{subject to system~\eqref{eq:systemmodel}}.
\end{equation}
We now present the last main result of the paper:

\begin{theorem}\label{thm:infinitehorizon}
    The following two items hold: 
    \begin{enumerate}
        \item If $f$ has a (unique) fixed point $W$, then for any $x\in \R^n$, problem~\eqref{eq:infiniteoptimal}
        admits a unique minimizer $u^*\in \mathcal{U}$ given by 
        \begin{equation}\label{eq:optimalcontrollawinfinitehorizon}
        {u}^*(t) = K(t-k\tau) {x}_{k\tau}, 
        \end{equation}
        for $t\in [k\tau, (k+1)\tau)$ and $k \in \N_0$, 
        where $K: [0,\tau) \to \R^{m\times n}$ is the stabilizing gain given in~\eqref{eq:stabilizingfeedbackgain}. The minimized cost is given by
        $$
        J^*(x):= J(x, u^*) = x^\top W x. 
        $$
        \item Conversely, if~\eqref{eq:infiniteoptimal} admits a solution for any $x\in \R^n$, then $f$ has a (unique) fixed point.  
    \end{enumerate}
\end{theorem}

\subsection{Numerical study}\label{ssec:numericalstudy}
In this subsection, we conduct a numerical study for feedback stabilization of system~\eqref{eq:systemmodel}, which is complementary to the above theoretical results and sheds light on the question posed at the end of Section~\ref{sec:intro}. 
We carry out two sets of simulations, the MATLAB codes of which are available at GitHub\footnote{Link to the codes:
\url{https://github.com/CynthiaC-yu/numerical-illustration}.}. For both settings, we let 
\begin{equation*}\label{eq:defAi}
A_i :=  \frac{(i-1)}{5}I +\frac{1}{10}\operatorname{diag}(0,1,2,3,4),
\end{equation*} 
for $i = 1,2,3,4,5$. 
None of these matrices is Hurwitz. We choose two different $B$ matrices, with $B = I$ and $B = \bfo$ (i.e., the vector of all ones). 
For the map $f$, we set $L = R = I$ for $B = I$ and $L = I$, $R = 1$ for $B = \bfo$.  
We initialize the iteration $f^N(P)$ with $P=0$. 
Then, for each pair $(A_i, B)$ and for each $\tau$ on a grid with spacing $0.2$, we search for the maximum $\sigma^2$ using the corresponding convergence criterion described below. We use binary search in $\sigma^2\in[0,4]$ with a relative
bracket-width tolerance of $2\%$ (i.e., at the terminating step, $(\sigma^2_{\rm high} - \sigma^2_{\rm low})/\sigma^2_{\rm high} \leq 2\%$,    
and we set $\sigma^2_{\max} = \sigma^2_{\rm low}$).
For $\tau>0$, we declare empirical convergence if  
$$\|f^{N}(0) - f^{N-1}(0)\| \leq \epsilon \|f^{N-1}(0)\|,\mbox{ with } \epsilon=10^{-5},$$ and empirical divergence if $$\|f^N(0)\| \geq 10^{12} \quad \mbox{or if} \quad N\geq 2\cdot 10^5.$$
All matrix norms are Frobenius norms. 
For the special case $\tau=0$, we solve the continuous Riccati equation
$$
\dot W(t)
= A^\top W(t)+W(t)A-C(W(t))+L,
\qquad W(0)=0,
$$
where
$$
C(W):=WB(R+\sigma^2B^\top WB)^{-1}B^\top W.
$$
We assess empirical convergence using the scaled derivative residual
$$
\rho(W(t))
:=
\frac{\|\dot W(t)\|_F}
{\max\left\{1,\,
\|A^\top W(t)\|_F+\|W(t)A\|_F
+\|C(W(t))\|_F+\|L\|_F
\right\}}.
$$
Empirical convergence is declared if $\rho(W(t))\leq\epsilon$, with $\epsilon=10^{-5}$, provided that $\|W(t)\|_F<10^{12}$. The maximum integration time is set to $5000$.

The corresponding $(\tau,\sigma^2_{\max})$ curves, for the two setups, are shown in Figure~\ref{sfig1:Bidentity} and Figure~\ref{sfig1:Bvecone}, respectively. 
We observe that for each $(A_i, B)$, $\sigma^2_{\max}$ is strictly monotonically decreasing in~$\tau$ . In words, with $(A_i, B)$ fixed, the larger the sampling rate is, the larger the allowable noise variance can be.
 
 \begin{figure}
    \centering
    \subfloat[\label{sfig1:Bidentity}]{
    \includegraphics[width=0.47\linewidth]{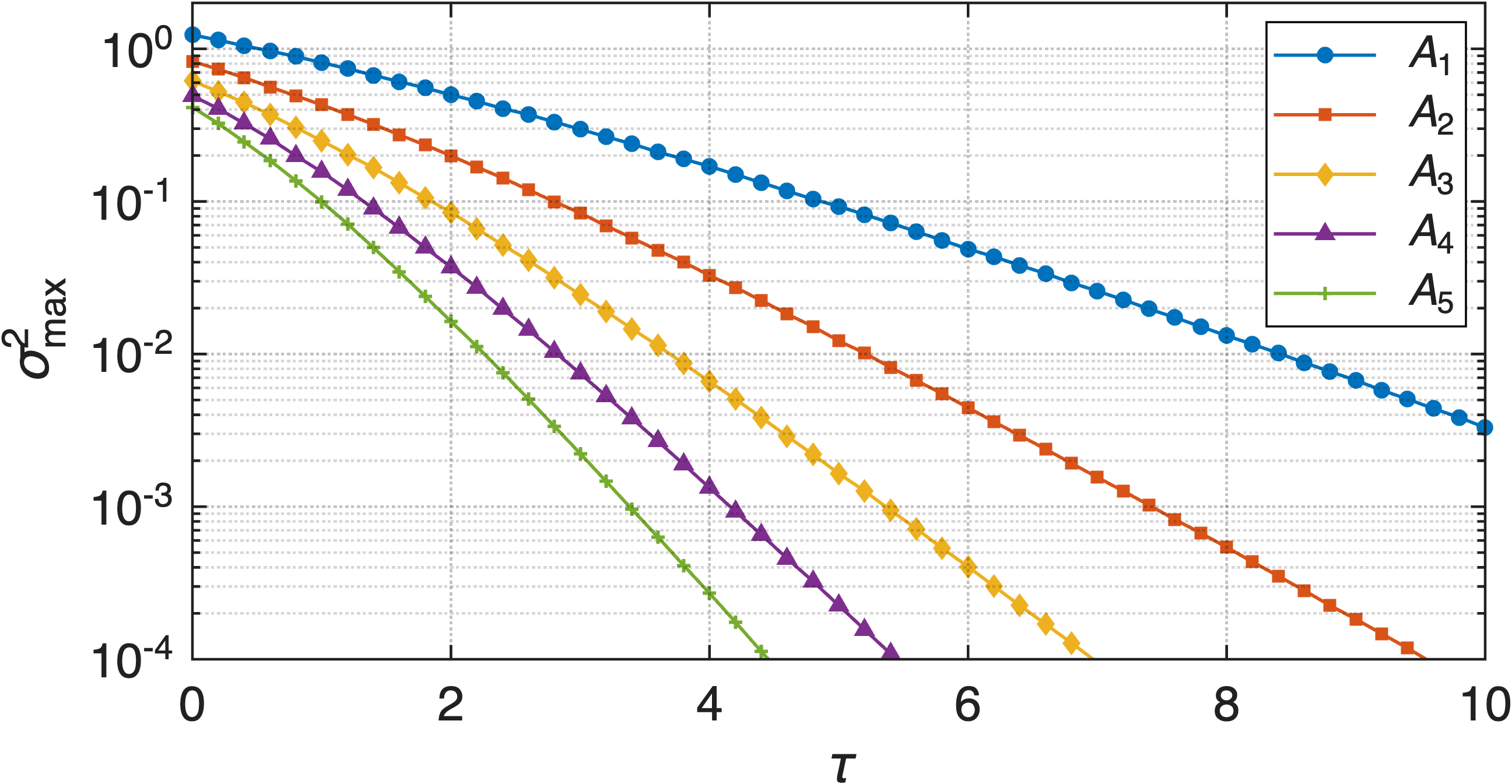}}\quad
    \subfloat[\label{sfig1:Bvecone}]{
    \includegraphics[width=0.47\linewidth]{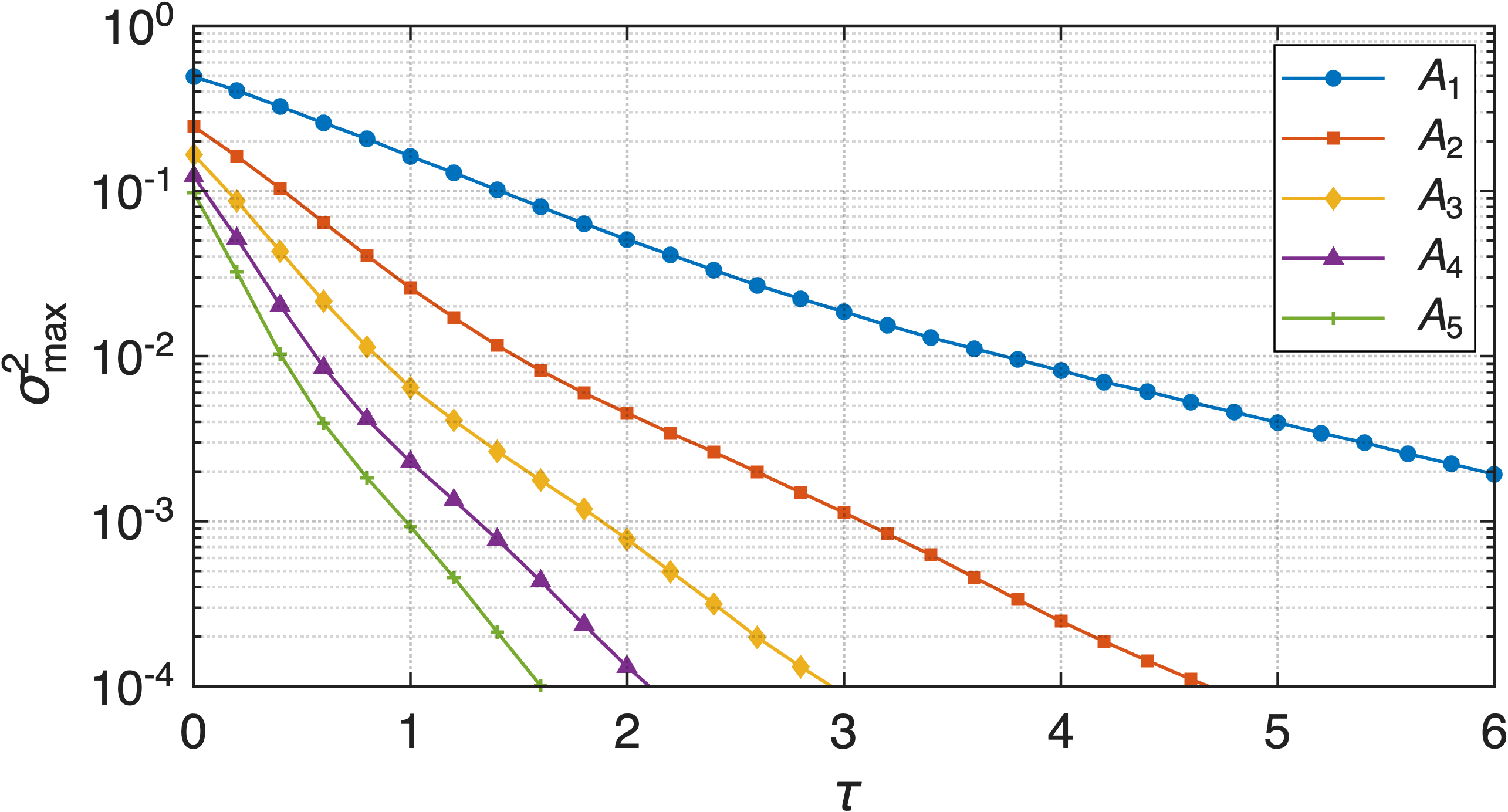}}
    \caption{\unboldmath 
    Plots of the $(\tau,\sigma^2_{\max})$ curves for $(A_i, B)$, with $B = I$ for~\ref{sfig1:Bidentity} and $B = \bfo$ for~\ref{sfig1:Bvecone}. For each $\tau$, we record $\sigma^2_{\max}$ using the corresponding convergence criterion described above. For the map $f$, we set $L = R =  I$ for $B = I$ and $L = I$, $R =  1$ for $B = \bfo$. The vertical axes use a logarithmic scale.}
    \label{fig:illustration}
\end{figure}

\section{Analysis and Proofs}\label{sec:proofs}
This section is dedicated to the proofs of the four theorems presented in Section~\ref{sec:mainresults}, and is organized as follows: 
\begin{enumerate}
\item First, in Subsection~\ref{ssec:finitehorizon}, we investigate the finite-horizon optimal control problem~\eqref{eq:finiteoptimal},  establish Theorem~\ref{thm:finitehorizon}, and validate Algorithm~\ref{alg:gainmatrix}. 
Part of the analysis serves as a cornerstone for proving the other theorems. 
\item Next, in Subsection~\ref{ssec:sufficiency}, we show that existence of a fixed point of $f$ is sufficient for mean-square feedback stabilizability of system~\eqref{eq:systemmodel} (item~1 of Theorem~\ref{thm:feedbackstabilizableandw}) and for existence and uniqueness of a solution to the optimal control problem~\eqref{eq:infiniteoptimal} (item~1 of Theorem~\ref{thm:infinitehorizon}). The proof does not rely on the uniqueness of the fixed point of~$f$. 
\item Then, in Subsection~\ref{ssec:necessity}, we establish the necessity part (i.e., item~2 of Theorem~\ref{thm:feedbackstabilizableandw} and item~2 of Theorem~\ref{thm:infinitehorizon}). 
\item Finally, in Subsection~\ref{ssec:convergence}, we establish Theorem~\ref{thm:globalconvergence}, the proof of which builds upon the connections between the map $f$,  feedback stabilizability of system~\eqref{eq:systemmodel}, and the solutions to the optimal control problems.    
\end{enumerate}

\subsection{Proof of Theorem~\ref{thm:finitehorizon}}\label{ssec:finitehorizon}

For ease of notation, we will suppress in this subsection the argument $N$ and simply write $W_k$, $Q_k(s)$, and $K_k(s)$. 
The proof builds upon dynamic programming. For any $k = 0,\ldots,N$, let $V^*_k(x)$ be the value-to-go function at time $k\tau$ with $x_{k\tau} = x$.

The boundary condition for the dynamic programming is simply 
$V_N^*(x) = x^\top P x$.  
The update rule is given by
\begin{equation}\label{eq:updaterule}
V^*_{k}(x)= \min_{u |_{[k\tau,(k+1)\tau)}} \Bigg [ \tau x^\top  L   x  + \int_{k\tau}^{(k+1)\tau} u^\top(t)  R  u(t) dt  + \bE \left [V^*_{k+1}(x_{(k+1)\tau})  \mid x_{k\tau} = x \right ]\Bigg ], 
\end{equation}
for any $k = 0,\ldots, N - 1$. 

Assuming that $V^*_{k+1}(x) = x^\top W_{k + 1} x $, we need to show  that 
$u^*_N |_{[k\tau, (k+1)\tau)}$ is the unique minimizer that solves the optimal control problem on the right-hand side of~\eqref{eq:updaterule}, where $u^*_N$ is given in the statement of Theorem~\ref{thm:finitehorizon} and, consequently, $V^*_{k}(x) = x^\top W_k x$.

We first express the term $\bE[V^*_{k+1}(x_{(k+1)\tau})]$ as an explicit function of $u|_{[k\tau,(k+1)\tau)}$.  
We need the following lemma (see the proof of Proposition~1 of~\cite{jenson2020optimal}; for completeness of presentation, we also include a relatively short proof below the statement of the lemma):

\begin{lemma}\label{lem:meanandconvariance}
    For any $t\in [k\tau, (k+1)\tau]$, let 
\begin{equation}\label{eq:defbarxandcovariance}
\begin{aligned}
\bar x(t) & := \bE[ x_{t} \mid x_{k\tau} = x ], \\
\Sigma(t) & := \mathrm{Cov}\left (x_t \mid x_{k\tau}=x \right ).
\end{aligned}
\end{equation}
Then, 
\begin{align}
\bar x((k+1)\tau)
&= e^{A\tau}x
+ \int_{k\tau}^{(k+1)\tau}
v(t) dt,
\label{eq:mean}\\
\Sigma((k+1)\tau)
&= \sigma^2 \int_{k\tau}^{(k+1)\tau}
v(t) v^\top(t) dt,
\label{eq:covariance}
\end{align}
where $v(t):= e^{A((k+1)\tau-t)} B u(t)$. 
\end{lemma}

\begin{proof}
It follows from the It{\^ o} rule and the expectation rule that $\bar x(t)$ and $\Sigma(t)$ obey the following linear ordinary differential equations: 
$$
\left\{
\begin{aligned}
\dot{\bar x}(t) & = A \bar x(t) + B u(t), \\
\dot \Sigma(t) & = A\Sigma(t) + \Sigma(t) A^\top + \sigma^2 B u(t)u^\top(t) B^\top,  
\end{aligned} 
\right.
$$
with $\bar x(k\tau)  = x$ and $\Sigma(k\tau)  = 0$, 
whose solutions at $t = (k+1)\tau$ are given by~\eqref{eq:mean} and~\eqref{eq:covariance}, respectively. 
\end{proof}

By the hypothesis that $V^*_{k+1}(x) = x^\top W_{k+1} x$ and Lemma~\ref{lem:meanandconvariance}, 
we obtain that
\begin{align}
    \bE[V^*_{k+1}(x_{(k+1)\tau}) \mid x_{k\tau}] 
    &  = \mathbb{E}\left[x_{(k+1)\tau}^{\top}W_{k+1}x_{(k+1)\tau} \mid x_{k\tau} \right] \notag \\
    &=\operatorname{tr}\left(W_{k+1}\mathbb{E}\left[x_{(k+1)\tau}x_{(k+1)\tau}^{\top} \mid x_{k\tau} \right]\right) \notag \\
    &=\operatorname{tr}\left(W_{k+1}\left(\bar x((k+1)\tau)\bar x^{\top}((k+1)\tau)+\Sigma((k+1)\tau)\right)\right) \notag \\
    & = \bar x^\top ((k+1)\tau) W_{k+1} \bar x((k+1)\tau)
    + \sigma^2\int_{k\tau}^{(k+1)\tau} v^\top(t) W_{k+1} v(t) dt, \label{eq:expectationofvstar}
\end{align}
where the third equality follows from the definitions of $\bar x((k+1)\tau)$ and $\Sigma((k+1)\tau)$ in~\eqref{eq:defbarxandcovariance} and the mean-variance decomposition and 
the last equality follows from~\eqref{eq:covariance}. Note that $\bar x((k+1)\tau)$ depends implicitly
on $u(s)$ as shown in~\eqref{eq:mean}.

Leveraging Lemma~\ref{lem:meanandconvariance} and~\eqref{eq:expectationofvstar}, we solve~\eqref{eq:updaterule} in the following lemma, which then concludes the proof of Theorem~\ref{thm:finitehorizon}:

\begin{lemma}\label{lem:optimalcontrol}
The optimal control problem~\eqref{eq:updaterule} has a unique minimizer, which is given by $u^*_N |_{[k\tau,(k+1)\tau)}$. 
Moreover, the minimal value is $V^*_k(x) = x^\top W_k x$. 
\end{lemma}

\begin{proof}
We consider all possible $\mathrm{L}^2$-integrable functions 
$\xi:[k\tau, (k+1)\tau)\to \R^m$ and show that $u^*_N |_{[k\tau,(k+1)\tau)}$ is still the unique minimizer in this possibly enlarged optimization space. 
First, note that the objective function is strictly convex, Fr\'echet differentiable, and coercive, so there exists at least one minimizer~ $\xi^*$. 
It follows from the first-order analysis of the objective on the right-hand side of~\eqref{eq:updaterule}, together with Lemma~\ref{lem:meanandconvariance} and~\eqref{eq:expectationofvstar}, that any minimizer~$\xi^*$ must satisfy the following condition: 
\begin{equation*}
\xi^*(t) = - Q_k^{-1}((k+1)\tau - t) B^\top e^{A^\top((k+1)\tau - t)} W_{k+1} \bar x((k+1)\tau),
\end{equation*}
where $Q_k$ is introduced in Algorithm~\ref{alg:gainmatrix}.   
Combining the above equation with~\eqref{eq:mean}, we obtain by computation that
\begin{equation}\label{eq:xstar}  
\bar x((k+1)\tau)  = [I + M_k W_{k+1}]^{-1} e^{A\tau} x,  
\end{equation}
where 
$$
M_{k} :=  \displaystyle\int_{0}^{\tau}  e^{As}B Q_k^{-1}(s) B^\top e^{A^\top s} ds. 
$$
Then, using~\eqref{eq:xstar} and the definition of $K_k$ in Algorithm~\ref{alg:gainmatrix}, we have that $\xi^* = u^*_N |_{[k\tau,(k+1)\tau)}$.  
  
It now remains to show that $V_k^*(x)=x^{\top}W_kx$. By~\eqref{eq:expectationofvstar} and with $v(t)$ now given by
$$v(t)= e^{A((k+1)\tau - t)} B u_N^*(t),$$ 
we have that
\begin{align*}
\mathbb{E}[V_{k+1}^*(x_{(k+1)\tau}) \mid x_{k\tau} = x]
\, = & \, x^{\top}e^{A^{\top}\tau}[I+W_{k+1}M_k]^{-1}W_{k+1}[I+M_kW_{k+1}]^{-1}e^{A\tau}x\\
&  +\sigma^2x^{\top}\int_0^\tau K_k^{\top}(s)B^{\top}e^{A^{\top}(\tau-s)}
W_{k+1}e^{A(\tau-s)}BK_k(s)\,ds\,x.
\end{align*}
Using~\eqref{eq:updaterule} and the above expression, we have that
\begin{align*}
V_k^*(x)&=x^{\top}\Bigg[\tau L+\int_0^\tau K_k^{\top}(s)
\left(R+\sigma^2B^{\top}e^{A^{\top}(\tau-s)}W_{k+1}e^{A(\tau-s)}B\right)K_k(s)\,ds\\
&\hspace{3em}+e^{A^{\top}\tau}[I+W_{k+1}M_k]^{-1}W_{k+1}[I+M_kW_{k+1}]^{-1}e^{A\tau}\Bigg]x\\
&=x^{\top}\Bigg[\tau L+\int_0^\tau K_k^{\top}(\tau-s)Q_k(s)K_k(\tau-s)\,ds\\
&\hspace{3em}+e^{A^{\top}\tau}[I+W_{k+1}M_k]^{-1}W_{k+1}[I+M_kW_{k+1}]^{-1}e^{A\tau}\Bigg]x\\
&=x^{\top}\Big[\tau L+e^{A^{\top}\tau}[I+W_{k+1}M_k]^{-1}W_{k+1}M_kW_{k+1}[I+M_kW_{k+1}]^{-1}e^{A\tau}\\
&\hspace{3em}+e^{A^{\top}\tau}[I+W_{k+1}M_k]^{-1}W_{k+1}[I+M_kW_{k+1}]^{-1}e^{A\tau}\Big]x\\
&=x^{\top}\left[\tau L+e^{A^{\top}\tau}[I+W_{k+1}M_k]^{-1}W_{k+1}e^{A\tau}\right]x=x^{\top}W_kx,
\end{align*}
where the second equality follows from the definition of $Q_k(s)$ and the third equality follows
from the definitions of $K_k(s)$ and $M_k$.
\end{proof}

\subsection{Proof of Sufficiency}\label{ssec:sufficiency}
In this subsection, we establish item~1 of Theorem~\ref{thm:feedbackstabilizableandw} and item~1 of Theorem~\ref{thm:infinitehorizon}. 
Specifically, we show that if the map $f$ has a fixed point $W$ (note again that any such $W$ is positive definite), then system~\eqref{eq:systemmodel} is mean-square feedback stabilizable and, moreover, the infinite-horizon optimal control problem~\eqref{eq:infiniteoptimal} has $u^*\in \mathcal{U}$ as the  unique minimizer. We will fix such a $W$ for the remainder of the subsection.

\subsubsection{Proof of item~1 of Theorem~\ref{thm:feedbackstabilizableandw}}

Let $K(s)$ be given as in~\eqref{eq:stabilizingfeedbackgain}, and $$u^*(t) = K(t - k\tau)x_{k\tau}, \quad \mbox{for } t\in [k\tau, (k+1)\tau) \mbox{ and } k\in \N_0,$$ 
be given as in~\eqref{eq:optimalcontrollawinfinitehorizon}. 
We show that $\lim_{t\to\infty}\mathbb{E}[\|x_t\|^2] = 0$.

We first show that 
$\lim_{k\to\infty} \mathbb{E}[\|x_{k\tau}\|^2] = 0$. 
For convenience, let $z_k := x^\top_{k\tau} W x_{k\tau}$.    
Let $\bar x((k+1)\tau)$ and $\Sigma((k+1)\tau)$ be given as in~\eqref{eq:defbarxandcovariance}, with expressions given in~\eqref{eq:mean} and~\eqref{eq:covariance}, respectively. Then, 
\begin{equation}\label{eq:meanvarianceseparation}
\bE{[z_{k+1}]} = \bE[\bar x^\top((k+1)\tau) W \bar x((k+1)\tau) + \tr(W\Sigma((k+1)\tau)) ],
\end{equation}
where the expectation on the right-hand side is with respect to $x_{k\tau}$.

We will now express the two terms inside the expectation on the right-hand side of~\eqref{eq:meanvarianceseparation} as quadratic forms in $x_{k\tau}$.
For the first term, we first use Lemma~\ref{lem:meanandconvariance} to obtain that
\begin{align*}
\bar x((k+1)\tau)&=e^{A\tau}x_{k\tau}+\int_{k\tau}^{(k+1)\tau}
e^{A((k+1)\tau-t)}BK(t-k\tau)\,dt\,x_{k\tau}\\
&=e^{A\tau}x_{k\tau}-\int_0^\tau e^{A(\tau-s)}BQ^{-1}(\tau-s)B^{\top}
e^{-A^{\top}s}\,ds\,(W-\tau L)x_{k\tau}\\
&=e^{A\tau}x_{k\tau}-M(W^{-1}+M)^{-1}e^{A\tau}x_{k\tau}\\
&=(I+MW)^{-1}e^{A\tau}x_{k\tau},
\end{align*}
where the second equality follows from the definition of $K$ and the third equality follows from the definition of $M$ given in~\eqref{eq:defmatrixMandQ} and the fact that
\begin{equation}\label{eq:WminusTauL}
W-\tau L=e^{A^\top \tau}(W^{-1}+M)^{-1}e^{A\tau},
\end{equation}
which holds because $W=f(W)$. It then follows that
\[
\bar x_{(k+1)\tau}^{\top}W\bar x_{(k+1)\tau}
=x_{k\tau}^{\top}e^{A^{\top}\tau}(W^{-1}+M)^{-1}W^{-1}(W^{-1}+M)^{-1}e^{A\tau}x_{k\tau}.
\]
For the second term, we use again Lemma~\ref{lem:meanandconvariance} to obtain that
\[
\Sigma((k+1)\tau)=\sigma^2\int_0^\tau e^{A(\tau-s)}BK(s)x_{k\tau}x_{k\tau}^{\top}
K^{\top}(s)B^{\top}e^{A^{\top}(\tau-s)}\,ds,
\]
which implies that
\begin{align*}
& \operatorname{tr}(W\Sigma((k+1)\tau)) \\
=\, & \sigma^2x_{k\tau}^{\top}\int_0^\tau K^{\top}(s)B^{\top}e^{A^{\top}(\tau-s)}
We^{A(\tau-s)}BK(s)\,ds\,x_{k\tau}\\
=\, & x_{k\tau}^{\top}\left[\int_0^\tau K^{\top}(\tau-s)Q(s)K(\tau-s)\,ds
-\int_0^\tau K^{\top}(s)RK(s)\,ds\right]x_{k\tau}\\
=\, & x_{k\tau}^{\top}\Bigg[e^{A^{\top}\tau}[W^{-1}+M]^{-1}M[W^{-1}+M]^{-1}e^{A\tau} -\int_0^\tau K^{\top}(s)RK(s)\,ds\Bigg]x_{k\tau},
\end{align*}
where the second equality follows from the definition of $Q(s)$ and the third equality follows
from the definitions of $K(s)$ and $M$ (the computation is similar to the one carried out at the end of the proof of Lemma~\ref{lem:optimalcontrol}).

Combining the above expressions with~\eqref{eq:meanvarianceseparation}, we obtain that
\begin{equation*}
\begin{aligned}
\mathbb{E}[z_{k+1}]
&=\mathbb{E}\left[x_{k\tau}^{\top}e^{A^{\top}\tau}(W^{-1}+M)^{-1}e^{A\tau}x_{k\tau}\right]
-\mathbb{E}\left[x_{k\tau}^{\top}\int_0^\tau K^{\top}(s)RK(s)\,ds\,x_{k\tau}\right]\\
&=\mathbb{E}[z_k]-\mathbb{E}\left[x_{k\tau}^{\top}\left(\tau L+
\int_0^\tau K^{\top}(s)RK(s)\,ds\right)x_{k\tau}\right],
\end{aligned}
\end{equation*}
where the second equality follows from~\eqref{eq:WminusTauL}.

It follows that the nonnegative sequence $\{\bE[z_k]\}_{k\in \N_0}$ is monotonically decreasing, so it converges to some nonnegative real number. In particular, the second term on the right-hand side of the above equation converges to $0$. 
Since the matrix $(\tau L + \int_0^\tau K^\top(s) R K(s) ds)$ is positive definite, $\lim_{k\to\infty} \bE[\|x_{k\tau}\|^2] = 0$.

We now show that $\lim_{t\to\infty}\bE[\|x_t\|^2] = 0$. Define two matrix-valued functions $G, H:[0,\tau]\to \R^{n\times n}$ as follows: 
\begin{equation}\label{eq:defGandH}
\begin{aligned}
G(s) & := e^{As} + \int_0^s e^{A(s - r)} BK(r) dr, \\
H(s) & := \sigma^2 \int_0^s K^\top(r) B^\top e^{A^\top (s - r)}e^{A(s - r)} B K(r) dr. 
\end{aligned}
\end{equation}
It is clear that $G(s)$ and $H(s)$ are uniformly bounded.  
By Lemma~\ref{lem:meanandconvariance}, we have that for any $t\in [k\tau, (k+1)\tau)$, 
$$
\bar x(t)  = G(t-k\tau) x_{k\tau} \mbox{ and }
\tr(\Sigma(t))  = x^\top_{k\tau} H(t-k\tau) x_{k\tau}.
$$
Furthermore, we have that
\begin{equation*}
\bE[\|x_{t}\|^2  ] = \bE [ \|\bar x(t)\|^2] + \tr(\bE [\Sigma(t)  ]) =  \bE[x^\top_{k\tau} (G^\top(t - k\tau) G(t-k\tau) + H(t - k\tau)) x_{k\tau}]. 
\end{equation*}
Since $G(s)$ and $H(s)$ are uniformly bounded and since 
$\lim_{k\to\infty}\bE[\|x_{k\tau}\|^2] = 0$, we conclude that  $\lim_{t\to\infty}\bE[\|x_t\|^2] = 0$.  
\hfill{\qed}

\subsubsection{Proof of item~1 of Theorem~\ref{thm:infinitehorizon}}
Let $\mathcal{U}^*$ be the subset of $\mathcal{U}$ such that if $u\in \mathcal{U}^*$, then $\lim_{t\to\infty}\bE[\|x_t\|^2] = 0$. The set $\mathcal{U}^*$ is nonempty because it contains the feedback control law $u^*$ as shown above.  
If a control law $u\in \mathcal{U}$ solves the infinite-horizon optimal control problem~\eqref{eq:infiniteoptimal}, then it necessarily belongs to $\mathcal{U}^*$. We fix any such $u$ and add to the cost $J(x, u)$ given in~\eqref{eq:costfunction} the following trivial identity:
$$
x_0^\top W x_0 - \sum_{k = 0}^\infty \left [ \bE[x_{k\tau}^\top W x_{k\tau}] -  \bE[x_{(k+1)\tau}^\top W x_{(k+1)\tau}] \right ] = 0.
$$
We then obtain 
\begin{equation*}
J(x, u) = x^\top W x + 
\sum_{k = 0}^\infty \bE\Bigg [ x_{k\tau}^\top (\tau L - W) x_{k\tau} + 
\int_{k\tau}^{(k+1)\tau} u^\top (t) R u(t) dt + \bE[x_{(k+1)\tau}^\top W x_{(k+1)\tau} \mid x_{k\tau}] 
\Bigg ],
\end{equation*}
where the outer expectation is with respect to $x_{k\tau}$. 
For convenience, let $a_k$ be the term in the bracket of the above equation, so we can write $J(x,u) = x^\top W x + \sum_{k = 0}^\infty \bE[a_k]$. 
With $x_{k\tau}$ fixed, $a_k$ is a function of $u|_{[k\tau, (k+1)\tau)}$. 
Using the same arguments in the proof of Lemma~\ref{lem:optimalcontrol} (with $W_k(N)$, $Q_k(s;N)$, and  $K_k(s;N)$ replaced by $W$, $Q(s)$, and $K(s)$, respectively), we obtain that  
$\arg\min_{\xi} a_k(\xi) = u^* |_{[k\tau, (k+1)\tau)}$, 
where the argument~$\xi$ is again taken from the possibly larger optimization space $\mathrm{L}^2([k\tau,(k+1)\tau),\R^m)$.  
Moreover, by computation, the minimal value is $a_k(u^* |_{[k\tau, (k+1)\tau)}) = 0$. 
It then follows that $J(x, u) \geq J(x, u^*) = x^\top W x$ and the equality holds if and only if $u = u^*$. 
\hfill{\qed}

\subsection{Proof of Necessity}\label{ssec:necessity}
In this subsection, we establish item~2 of Theorem~\ref{thm:feedbackstabilizableandw} and item~2 of Theorem~\ref{thm:infinitehorizon}. The proof relies on two key lemmas as we outline below:
\begin{enumerate}
    \item Recall that $W_0(N) = f^N(P)$. We show in Lemma~\ref{lem:monotonicityofW} that if $P = 0$, then $W_0(N)$ is monotonically increasing in $N$. We will see soon that item~2 of Theorem~\ref{thm:infinitehorizon} is an immediate consequence of the lemma. 
    \item In Lemma~\ref{lem:finitecost}, we show that if system~\eqref{eq:systemmodel} is mean-square feedback stabilizable, then $\{W_0(N)\}_{N\in \N_0}$ is uniformly bounded above. This upper bound and the monotonicity established in Lemma~\ref{lem:monotonicityofW} yield item~2 of Theorem~\ref{thm:feedbackstabilizableandw}.    
\end{enumerate}

\subsubsection{Proof of item~2 of Theorem~\ref{thm:infinitehorizon}}
We use the following lemma (the arguments are standard, and we include a short proof for completeness of presentation):

\begin{lemma}\label{lem:monotonicityofW}
   The sequence $\{f^N(0)\}_{N\in \N_0}$ is monotonically increasing.
\end{lemma}

\begin{proof}
Recall that $u^*_N$ is the optimal control law that minimizes the cost function $J(x,u;N)$ and that $J^*(x;N)= J(x,u^*_N; N)$. 
To relate $J^*(x;N)$ and $J^*(x;N + 1)$, we consider the optimal control law $u^*_{N + 1}$ for the latter, and define  
$$\hat u_N := u^*_{N+1}|_{[0,N\tau)}.$$ 
On the one hand, since $P = 0$, we have that
\begin{equation*}
 J^*(x;N+1) - J(x, \hat u_N; N) = 
\bE \left [ \tau x^\top_{N\tau}\, L \, x_{N\tau} + \int_{N\tau}^{(N+1)\tau} u^{*\top}_{N + 1}(t) R u^*_{N + 1}(t) dt  \mid x_0 = x\right] \geq 0.
\end{equation*}
On the other hand, by optimality of $u^*_N$, we have that
$$J(x, \hat u_N; N) \geq J^*(x; N).$$ 
Combining the above two inequalities, we obtain that
$$
x^\top W_0(N + 1) x =  J^*(x; N + 1) \geq J^*(x; N) = x^\top W_0(N) x, 
$$
which holds for all $x\in \R^n$, so $W_0(N + 1) \geq W_0(N)$.  
\end{proof}

Now, suppose that for any given $x\in \R^n$ the infinite-horizon optimal control problem~\eqref{eq:infiniteoptimal} has a solution $u'\in \mathcal{U}$; then, by the same arguments in the proof of Lemma~\ref{lem:monotonicityofW}, we have that
$$
x^\top W_0(N) x \leq J(x,u') < \infty,   
$$
for all $N\in \N_0$. 
By Lemma~\ref{lem:monotonicityofW} and the monotone convergence theorem, $\lim_{N\to\infty} x^\top W_0(N) x$ exists for all $x\in \R^n$ and hence the following limit exists $$W:= \lim_{N\to\infty} W_0(N) = \lim_{N\to\infty} f^N(0),$$ 
which is necessarily a fixed point of~$f$.   \hfill{\qed}
 
\subsubsection{Proof of item~2 of Theorem~\ref{thm:feedbackstabilizableandw} }
We assume that system~\eqref{eq:systemmodel} is mean-square feedback stabilizable. 
Let $\tilde K: [0, \tau)\to \R^{m\times n}$ be a stabilizing feedback gain, and define $\tilde u(t):= \tilde K(t - k\tau) x_{k\tau}$, for $t\in [k\tau, (k+1)\tau)$ and $k\in \N_0$.  
We need the following lemma: 

\begin{lemma}\label{lem:finitecost}
    There exists a $\tilde W > 0$ such that for any $ x\in \R^n$, 
    \begin{equation*}
    \tilde J(x) := J(x, \tilde u) = x^\top \tilde W x.
    \end{equation*} 
\end{lemma}

\begin{proof} 
Let $\tilde G(s)$ be defined in the same way as $G(s)$ in~\eqref{eq:defGandH}, but with $K$ replaced by $\tilde K$. 
Recall that $\bS_n$ is the space of $n$-by-$n$ symmetric matrices. We equip $\bS_n$ with the inner product $\langle X, Y \rangle := \tr(XY)$. Define the map $\mathcal{L}:\bS_n\to \bS_n$ as
\begin{equation*}
\mathcal{L}(S) = \tilde G(\tau) S \tilde G^\top(\tau) +  \sigma^2 \int_0^\tau e^{A(\tau - s)} B \tilde K(s) S \tilde K^\top (s) B^\top e^{A^\top (\tau - s)} ds.
\end{equation*}
Let $\mathcal{L}^*$ be the dual of $\mathcal{L}$, which can be expressed explicitly as 
\begin{equation*}
\mathcal{L}^*(S) = \tilde G^\top(\tau) S  \tilde G(\tau) +  \sigma^2 \int_0^\tau \tilde K^\top (s) B^\top e^{A^\top (\tau - s)} S e^{A(\tau - s)} B \tilde K(s)  ds.
\end{equation*}

Now, consider the stochastic system~\eqref{eq:systemmodel} driven by $\tilde u(t)$. 
Let $\bar x((k+1)\tau)$ and $\Sigma((k+1)\tau)$ be defined as in~\eqref{eq:defbarxandcovariance}. 
Using Lemma~\ref{lem:meanandconvariance}, we obtain that 

\begin{align*}
\bar x((k+1)\tau)
&=\left[e^{A\tau}+\int_0^\tau e^{A(\tau-s)}B\tilde K(s)\,ds\right]x_{k\tau}
=\tilde G(\tau)x_{k\tau},\\
\Sigma((k+1)\tau)
&=\sigma^2\int_0^\tau e^{A(\tau-s)}B\tilde K(s)x_{k\tau}x_{k\tau}^{\top}
\tilde K^{\top}(s)B^{\top}e^{A^{\top}(\tau-s)}\,ds.
\end{align*}
It follows that
\begin{equation*}
\mathbb E[x_{(k+1)\tau}x_{(k+1)\tau}^{\top}]
=\mathbb E[\bar x_{(k+1)\tau}\bar x_{(k+1)\tau}^{\top}]+\mathbb{E}[\Sigma((k+1)\tau)]=
\mathcal L(\mathbb E[x_{k\tau}x_{k\tau}^{\top}]),
\end{equation*}
so $\bE[x_{k\tau}x_{k\tau}^\top] = \mathcal{L}^k(x_0x_0^\top)$ for all $k\in \N_0$.

Since $\tilde K$ is a stabilizing feedback gain,  $\bE[x_{k\tau}x_{k\tau}^\top]$ converges to $0$ as $k\to\infty$ for any initial condition $x_0\in \R^n$. 
Note that the space $\bS_n$ is spanned by $xx^\top$ for all $x\in \R^n$. 
Thus,  $\mathcal{L}$ is a stable linear operator, i.e., all of its eigenvalues belong to the interior of the unit disk of the complex plane. 

Finally, we compute $\tilde J(x)$ and show that it can be expressed as a quadratic form. For convenience, let 
$$
Y:= \tau L  + \int_0^\tau \tilde K^\top(s) R \tilde K(s) ds.
$$
Then, 
\begin{align*}
\tilde J(x)  
& =\sum_{k=0}^{\infty}\mathbb E\left[\tau x_{k\tau}^{\top}Lx_{k\tau}
+\int_{k\tau}^{(k+1)\tau}\tilde u^{\top}(t)R\tilde u(t)\,dt\mid x_0=x\right]\\
& = \sum_{k = 0}^\infty\tr\left ( Y\bE[x_{k\tau} x_{k\tau}^\top] \right ) 
= \sum_{k = 0}^\infty \tr \left ( Y \mathcal{L}^k(xx^\top) \right ) 
= x^\top \sum_{k = 0}^\infty \mathcal{L}^{*k}(Y)  x.
\end{align*}
Since $\mathcal{L}^*$ and $\mathcal{L}$ share the same eigenvalues, all eigenvalues of $\mathcal{L}^*$ belong to the interior of the unit disk of the complex plane. We thus conclude that $\tilde W:= \sum_{k = 0}^\infty \mathcal{L}^{*k}(Y)$ exists. 
\end{proof}

Item~2 of Theorem~\ref{thm:feedbackstabilizableandw} follows directly from Lemmas~\ref{lem:monotonicityofW} and~\ref{lem:finitecost}; indeed, for any $x\in \R^n$, $\{x^\top f^N(0) x\}_{N\in \N_0}$ is monotonically increasing in $N$ and is bounded above by $x^\top \tilde W x$. Thus, $W:=\lim_{N\to\infty}f^N(0)$ exists and satisfies $W = f(W)$.  \hfill{\qed}

\subsection{Proof of Theorem~\ref{thm:globalconvergence}}\label{ssec:convergence}
In this subsection, we show that if $f$ has a fixed point $W$, then for any $P\geq 0$, $\lim_{N\to\infty}f^N(P) = W$. This, in particular, implies that $W$ is the unique fixed point of $f$; indeed, if $W'$ is another fixed point of $f$, then $W' = \lim_{N\to\infty}f^N(W') = W$.  

Since $f$ has a fixed point, system~\eqref{eq:systemmodel} is mean-square feedback stabilizable as shown in Subsection~\ref{ssec:sufficiency}. Then, by the arguments in Subsection~\ref{ssec:necessity}, we have that 
$f^N(0)$ converges to a fixed point of $f$. Let the fixed point $W$ be chosen such that $$W = \lim_{N\to\infty} f^N(0).$$   

Let $J(x, u;N)$ be the cost function given in~\eqref{eq:finitehorizon} corresponding to $P = 0$, and $J'(x, u;N)$ correspond to an arbitrary $P \geq 0$, which will be fixed in the sequel.  
Let $J^*(x; N)$ and $J'^*(x; N)$ be defined in the same way, and let  $u^*_N(x)$ and $u'^*_N(x)$ be the associated optimal control laws.

Let $u^*\in \mathcal{U}$ be the optimal control law given in~\eqref{eq:optimalcontrollawinfinitehorizon}, which solves the infinite-horizon optimal control problem~\eqref{eq:infiniteoptimal}. 
We consider the stochastic system~\eqref{eq:systemmodel} driven by $u^*(t)$. All the expectations below are with respect to solutions of this system.

Given any $\epsilon > 0$, we show below that for sufficiently large $N$, $|x^\top (f^N(P) - W) x| \leq \epsilon \|x\|^2$ for any $x\in \R^n$. 

First, by item~1 of Theorem~\ref{thm:infinitehorizon}, we have that $J^*(x) = x^\top W x$. Since the feedback gains $K_k$ associated with $u^*$ are time invariant, i.e., $K_k = K$ for all $k\in \N_0$, it follows that
\begin{equation}\label{eq:diff1}
\delta_1(N) := J^*(x) - J(x, u^*|_{[0,N\tau)}; N) = \bE[x_{N\tau}^\top W x_{N\tau}].
\end{equation}

Next, note that  $J'(x, u;N) \geq J(x,u; N)$ for all $x\in \R^n$, for all $u \in \mathcal{U}_N$, and for all $N\in \N_0$. This, in particular, implies that $J'^*(x; N) \geq J^*(x; N)$.  Combining this with the fact that  $J(x, u^*|_{[0,N\tau)}; N) \leq J^*(x)$, we obtain that 
\begin{equation}\label{eq:diff2}
\delta_2(N)= J'^*(x; N)  - J(x, u^*|_{[0,N\tau)}; N) \geq  J^*(x;N) - J^*(x) = x^\top (f^N(0) - W) x. 
\end{equation}
We can also bound $\delta_2(N)$ from above and have that 

\begin{equation}\label{eq:diff3}
\delta_2(N) \leq   J'(x, u^*|_{[0,N\tau)}; N) - J(x, u^*|_{[0,N\tau)}; N)  = \bE[x_{N\tau}^\top P x_{N\tau}],
\end{equation}
where the inequality follows from the optimality of $u'^*_N$.

By the same argument in the proof of Lemma~\ref{lem:finitecost}, we have that $\bE[\|x_{N\tau}\|^2]$ decays exponentially fast in~$N$. 

Thus, given any $\epsilon>0$, there exists $N(\epsilon)\in\N_0$ such that, for all $N\geq N(\epsilon)$ and all $x\in\R^n$,
\begin{equation}\label{eq:diff4}
\bE[x_{N\tau}^\top(W+P)x_{N\tau}]
\leq \frac{\epsilon}{2}\|x\|^2.
\end{equation}

Also, since $f^N(0)$ is monotonically increasing in $N$ and converges to $W$, we can increase $N(\epsilon)$, if necessary, so that 
\begin{equation}\label{eq:diff5}
0 \leq W - f^N(0) \leq \frac{\epsilon}{2} I \quad \mbox{for all } N \geq N(\epsilon).
\end{equation}

Combining the above arguments, we obtain that for any $x\in \R^n$ and for any $N\geq N(\epsilon)$, 
\begin{align*}
|x^\top (f^N(P) - W) x |  
& =   |J'^*(x; N) - J^*(x)| \\
& \leq  |\delta_1(N)| + |\delta_2(N)| \\
& \leq  \bE[x_{N\tau}^\top (W +P) x_{N\tau}] + x^\top (W - f^N(0)) x \leq  \epsilon \|x\|^2,
\end{align*}
where the second inequality follows from~\eqref{eq:diff1}, ~\eqref{eq:diff2}, and~\eqref{eq:diff3} and the last inequality follows from~\eqref{eq:diff4} and~\eqref{eq:diff5}. We thus conclude that $\lim_{N\to\infty}f^N(P) = W$. \hfill{\qed}
 
\printbibliography

\appendix

\section{Numerical study for different thresholds $\epsilon$}
For $\tau>0$, recall that we have used the following stopping criterion: We declare empirical convergence if $\|f^{N}(0) - f^{N-1}(0)\| \leq \epsilon \|f^{N-1}(0)\|$ with $\epsilon = 10^{-5}$, and empirical divergence if $\|f^N(0)\| \geq 10^{12}$ or if $N\geq 2\cdot 10^5$ (the two parameters $10^{12}$ and $2\cdot 10^5$ can be viewed as computational limits, which we fix in our simulations). 
We have also run simulations for other values of $\epsilon$. %including $\epsilon = 10^{-3}, 10^{-4}, 10^{-6}$. %The resulting $(\tau, \sigma^2_{\max})$ curves are nearly indistinguishable from the ones presented in Figure~\ref{fig:illustration}. 
We provide in Figure~\ref{fig:comparison} the $(\tau,\sigma^2_{\max})$ curves for the pairs $(A_i, I)$ and $(A_i, \mathbf{1})$, with $i = 1,2,3,4,5$ and $\epsilon = 10^{-3}, 10^{-4}, 10^{-5}, 10^{-6}$, plotted in~\ref{sfig1:BI_comparison} and~\ref{sfig1:B1_comparison} of Figure~\ref{fig:comparison}, respectively. 
For each tested pair $(A_i,B)$, the maximum observed relative difference among the four tolerance settings is less than $2.10\%$,  using the $\epsilon = 10^{-6}$ result as the reference.  
The simulation codes are available on GitHub (the link is provided in footnote~1)
for reproducibility of our results. 
For $\tau=0$, we use the scaled derivative residual described
in Subsection~\ref{ssec:numericalstudy}.

\begin{figure}[hbt]
    \centering
    \subfloat[\label{sfig1:BI_comparison}]{
    \includegraphics[width=0.68\linewidth]{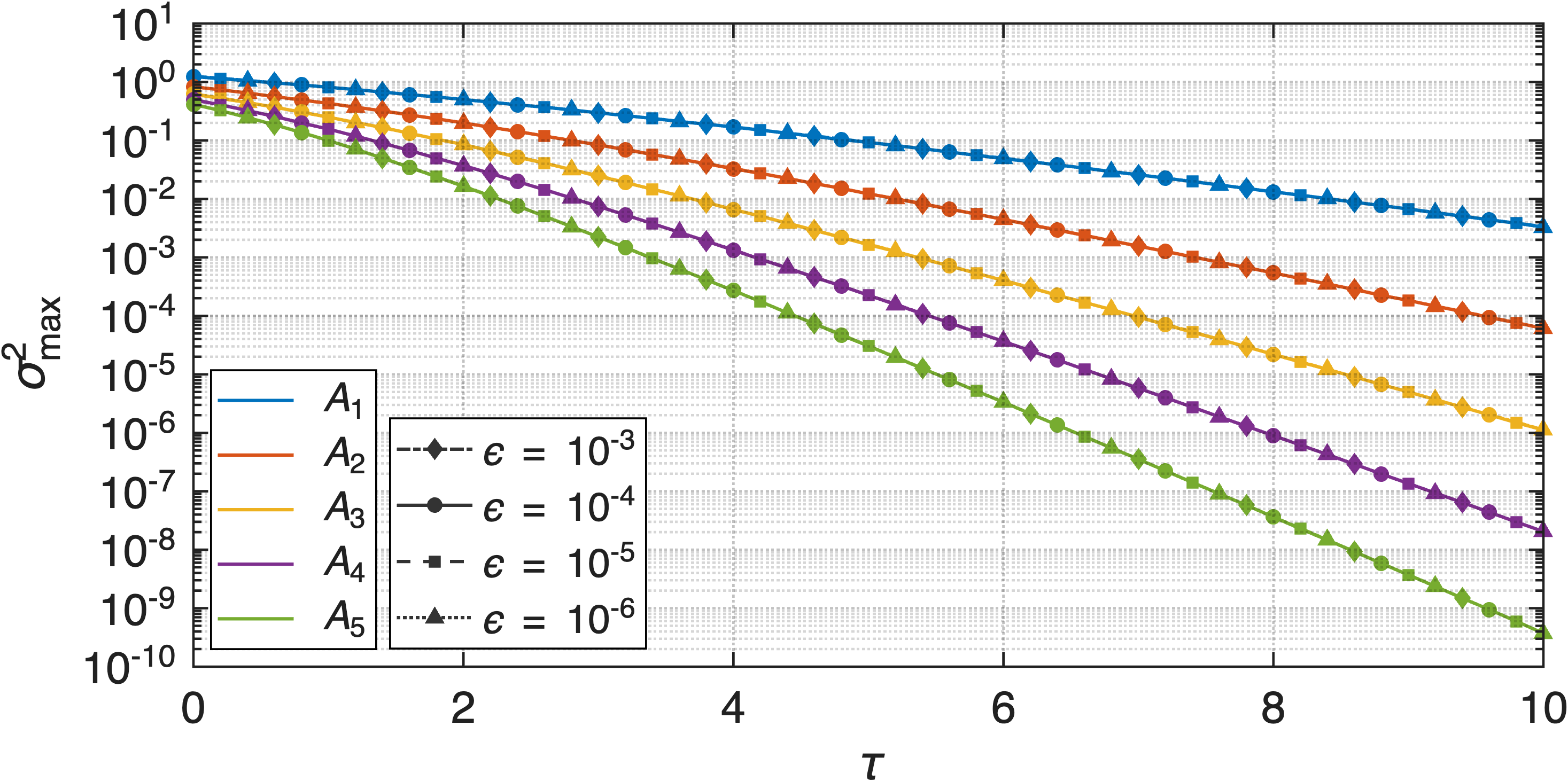}}
    
    \subfloat[\label{sfig1:B1_comparison}]{
    \includegraphics[width=0.68\linewidth]{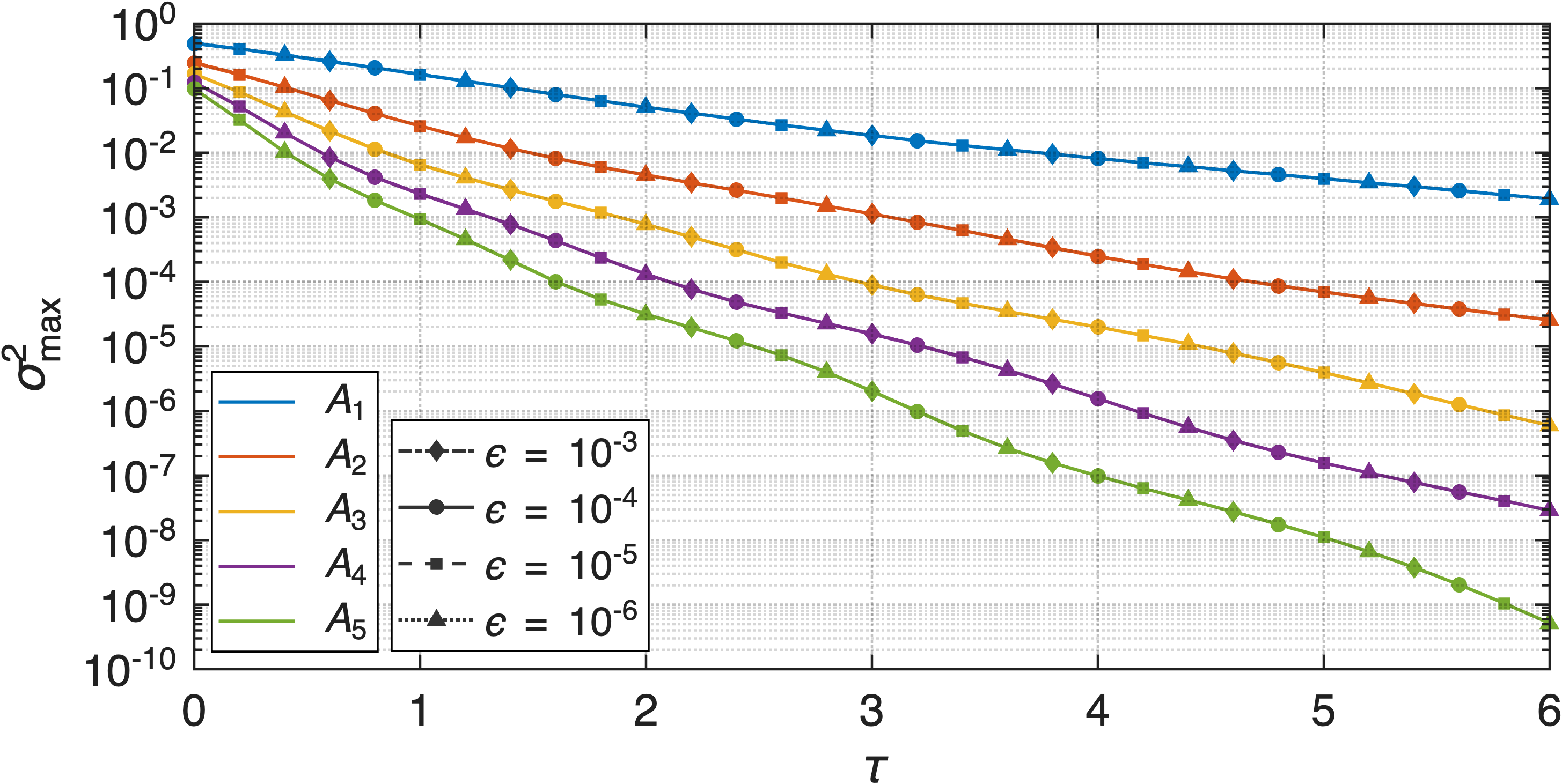}}
    \caption{\unboldmath
    Plots of the $(\tau, \sigma^2_{\max})$ curves for the pairs $(A_i, B)$ given in Subsection~\ref{ssec:numericalstudy}, with different parameters $\epsilon = 10^{-j}$ for $j = 3,4,5,6$. 
    The figure in~\ref{sfig1:BI_comparison} is for the case $B = I$,  with $L = R =  I$, and the figure in~\ref{sfig1:B1_comparison} is for the case  $B = \bfo$, with $L = I$ and $R = 1$. We initialize the iteration $f^N(P)$ with $P=0$ for both cases. The vertical axes are plotted in the log-scale. Each color corresponds to a matrix $
    A_i$, for $i = 1,\ldots, 5$. To each color, there correspond four curves, plotted with different line styles but on the same $\tau$ grid, for the four different values of~$\epsilon$.  
    The four curves, for each $(A_i, B)$, are nearly indistinguishable; their markers are thus staggered to make them easy to read.}
    \label{fig:comparison}
\end{figure}

\end{document}